\documentclass[11pt,leqno]{amsart}

\usepackage{amsmath,amssymb,amsthm,mathrsfs}
\usepackage{hyperref}
\usepackage{tikz} % Colors (for more, see http://latexcolor.com/)

\newtheorem{theorem}{Theorem}[section]
\newtheorem*{theorem*}{Theorem}

\newtheorem{question}[theorem]{Question}

\newtheorem{proposition}[theorem]{Proposition}
\newtheorem{corollary}[theorem]{Corollary}

\theoremstyle{definition}
\newtheorem{definition}[theorem]{Definition}

\theoremstyle{remark}
\newtheorem{remark}[theorem]{Remark}
\numberwithin{equation}{section}

\DeclareMathOperator{\conv}{conv}
\DeclareMathOperator{\LASQ}{LASQ}
\DeclareMathOperator{\positiveLASQ}{positive\ LASQ}
\DeclareMathOperator{\WASQ}{WASQ}
\DeclareMathOperator{\ASQ}{ASQ}
\DeclareMathOperator{\positiveASQ}{positive\ ASQ}
\DeclareMathOperator{\LDtwoP}{LD2P}
\DeclareMathOperator{\DtwoP}{D2P}
\DeclareMathOperator{\SDtwoP}{SD2P}
\begin{document}
\title{Locally almost square Banach lattices}	
\author{\href{https://orcid.org/0000-0002-1500-8123}{Stefano Ciaci}}
\address{Institute of Mathematics and Statistics, University of Tartu, Narva mnt 18, 51009 Tartu, Estonia}
\email{stefano.ciaci@ut.ee}
\urladdr{\url{https://stefanociaci.science.blog/}}
\thanks{This work was supported by the Estonian Research Council grant (PSG487).}
%==========================================================
\begin{abstract} 
    A Banach space is \textit{locally almost square} if, for every $y$ in its unit sphere, there exists a sequence $(x_n)$ in its unit sphere such that $\lim\|y\pm x_n\|=1$. A Banach space is \textit{weakly almost square} if, in addition, we require the sequence $(x_n)$ to be weakly null. It is known that these two properties are distinct, so we aim to investigate if local almost squareness implies a weaker version of the latter property by replacing the sequence with a net. In order to achieve this result, we restrict ourselves to Banach lattices and introduce a strengthening of local almost squareness by requiring that the sequence is in the positive cone of the lattice.

    As an application of such characterization, we prove that this positive variant of local almost squareness implies that every relatively weakly open set in the unit ball has diameter 2, that is, the Banach space has the so called \textit{diameter two property}. This in particular allows us also to generate new examples of Banach spaces enjoying the diameter two property.
\end{abstract}

\maketitle

\section{Introduction}\label{Section: Introduction}

Let $X$ be a Banach space.
\begin{enumerate}
    \item $X$ is \textbf{locally almost square} ($\LASQ$) if, for every $y\in S_X$ and $\varepsilon>0$, there exists $x\in S_X$ satisfying $\|y\pm x\|\le1+\varepsilon$.
    \item $X$ is \textbf{weakly almost square} ($\WASQ$) if, for every $y\in S_X$, there exists a weakly null sequence $(x_n)$ in $X$ satisfying $\lim\|y\pm x_n\|=1$ and $\lim\|x_n\|=1$.
    \item $X$ is \textbf{almost square} ($\ASQ$) if, for every $y_1,\ldots,y_n\in S_X$ and $\varepsilon>0$, there exists $x\in S_X$ satisfying $\|y_i\pm x\|\le1+\varepsilon$ for all $1\le i\le n$.
\end{enumerate}

These properties were defined in \cite{ALL2016}. In the same paper the authors also proved that being $\ASQ$ implies being $\WASQ$, on the other hand the fact that being $\WASQ$ implies being $\LASQ$ is trivial. It is known that $L_1[0,1]$ is $\WASQ$ but it is not $\ASQ$, but the fact that being $\LASQ$ is not equivalent to being $\WASQ$ was just proved in the recent preprint \cite{KV2022-preprint}. This last result is the starting point of our investigation.

In this paper we aim to prove that being $\LASQ$ implies possessing a property which is very close to being $\WASQ$, namely by substituting the sequence with a net. In order to do so, we have to restrict ourselves to the class of Banach lattices and we introduce a slight strengthening of being $\LASQ$ in which we additionally require the ``M-orthogonal element'' to be in the positive cone of the lattice. We call this property $\positiveLASQ$ (see Definition~\ref{Definition: LASQ+}). Before explaining more in detail the content of the paper let us recall that a slice of the unit ball $B_X$ is the intersection of $B_X$ with some half-space and let us remind of a few more properties that a Banach space $X$ can enjoy.
\begin{enumerate}
    \item $X$ has the \textbf{local diameter two property} ($\LDtwoP$) if every slice of $B_X$ has diameter 2.
    \item $X$ has the \textbf{diameter two property} ($\DtwoP$) if every non-empty relatively weakly open set in $B_X$ has diameter 2.
    \item $X$ has the \textbf{strong diameter two property} ($\SDtwoP$) if every convex combination of slices in $B_X$ has diameter 2.
\end{enumerate}

These properties were introduced in \cite{ALN2013} and it is known that if a Banach space is $\LASQ$, then it has the $\LDtwoP$ \cite[Proposition~2.5]{K2014}, if it is $\WASQ$, then it has the $\DtwoP$ \cite[Proposition~2.6]{K2014} and if it is $\ASQ$, then it has the $\SDtwoP$ \cite[Proposition~2.5]{ALL2016}.

\subsection{Content of the paper}\label{Subsection: Content of the paper}\hfill

In Section~\ref{Section: Main result} we prove the main result of the paper.

\begin{theorem*}
    Let $X$ be a Banach lattice. If $X$ is $\positiveLASQ$, then, for every $y\in S_X$, there exists a weakly null net $(x_\alpha)$ in $X$ satisfying
    \begin{equation*}
        \lim\|x_\alpha\|=1\text{ and }\lim\|y\pm x_\alpha\|=1.
    \end{equation*}
\end{theorem*}

In addition, using the same technique, we prove that, under the additional assumption that $X^*$ is separable, if $X$ is $\positiveLASQ$, then $X$ is $\WASQ$ (see Theorem~\ref{Theorem: LASQ implies WASQ}).

In Section~\ref{Section: The diameter two property} we use the main result to prove that being $\positiveLASQ$ implies enjoying the $\DtwoP$ (see Theorem~\ref{Theorem: LASQ implies D2P}). From this result, we extend known stability results for $\LASQ$ spaces to $\positiveLASQ$ spaces and, as a consequence, we provide new examples of Banach spaces which enjoy the $\DtwoP$. Namely, we show that, if $X$ is $\positiveLASQ$, then both Köthe-Bochner spaces $E(X)$ and ultrapowers $X^\mathscr F$ enjoy the $\DtwoP$ (see Corollaries~\ref{Corollary: Köthe-Bochner spaces} and \ref{Corollary: ultrapowers}). Moreover, if $X$ is $\positiveASQ$ and $Y$ is an AL-space, then the projective tensor product $X\hat{\otimes}_\pi Y$ has the $\DtwoP$ (see Corollary~\ref{Corollary: projective tensor product}).

\subsection{Notation}\label{Subsection: Notation}\hfill

We consider only real Banach lattices. For a Banach lattice $X$ we denote by $B_X$ its closed unit ball, by $S_X$ its unit sphere and we denote by $S^+_X$ the intersection of $S_X$ with the positive cone of the lattice. The rest of the notation for Banach spaces and lattices is standard and we refer the reader to \cite{AB2006}.

\section{Main result}\label{Section: Main result}

\begin{definition}\label{Definition: LASQ+}
    We say that a Banach lattice $X$ is $\positiveLASQ$ if, for every $y\in S_X$ and $\varepsilon>0$, there exists $x\in S^+_X$ satisfying $\|y\pm x\|\le1+\varepsilon$.
\end{definition}

Notice that being $\positiveLASQ$ trivially implies being $\LASQ$. However, $L_1[0,1]$ is $\LASQ$ but it cannot be $\positiveLASQ$, since it is clear that no AL-space can be $\positiveLASQ$.

\begin{remark}\label{Remark: equivalent definition LASQ+}
    In Definition~\ref{Definition: LASQ+} we can require only that $\|y+x\|\le1+\varepsilon$ and $y\in S^+_X$. This is equivalent because, for every $y\in S_X$ and $\varepsilon>0$, there exists $x\in S^+_X$ satisfying $\||y|+x\|\le1+\varepsilon$, and, since $|y\pm x|\le|y|+x$, then
    \begin{equation*}
        \|y\pm x\|\le\||y|+x\|\le1+\varepsilon.
    \end{equation*}
\end{remark}

The following proposition gives us a class of Banach lattices which are $\positiveLASQ$.

\begin{proposition}\label{Proposition: AM-space LASQ implies LASQ+}
    Let $X$ be an AM-space. If $X$ is $\LASQ$, then $X$ is $\positiveLASQ$.
\end{proposition}
\begin{proof}
    Let $y\in S_X$ and $\varepsilon>0$. Since $X$ is $\LASQ$, there exists $x\in S_X$ satisfying $\|y\pm x\|\le1+\varepsilon$. Notice that
    \begin{equation*}
        |y+|x||=|y+(x\vee-x)|=|(y+x)\vee(y-x)|\le|y+x|\vee|y-x|,
    \end{equation*}
    therefore
    \begin{equation*}
        \|y+|x|\|\le\||y+x|\vee|y-x|\|=\max\{\|y+x\|,\|y-x\|\}\le1+\varepsilon.
    \end{equation*}
\end{proof}

Let us now state and prove the main result of the paper.

\begin{theorem}\label{Theorem: Main result}
    Let $X$ be a Banach lattice. If $X$ is $\positiveLASQ$, then, for every $y\in S_X$, there exist a weakly null net $(x_\alpha)$ in $X$ satisfying
    \begin{equation*}
        \lim\|x_\alpha\|=1\text{ and }\lim\|y\pm x_\alpha\|=1.
    \end{equation*}
\end{theorem}
\begin{proof}
    Fix $y\in S_X$ and let us temporarily fix some $m\in\mathbb N$. Find a sequence $(\varepsilon_n^m)$ of strictly positive reals such that $\prod(1+\varepsilon_n^m)\le1+m^{-1}$. We now construct a sequence $(y_n^m)\subset S^+_X$ by recursion over $n\in\mathbb N$.
    
    At first define $y_0^m:=|y|$. Now assume that we are given $y_0^m,\ldots,y_n^m$ and set
    \begin{equation*}
        \hat{y}_n^m:=\frac{\sum_{i=0}^ny_i^m}{\left\|\sum_{i=0}^ny_i^m\right\|}\in S^+_X
    \end{equation*}
    and, since $X$ is $\positiveLASQ$, we can find $y_{n+1}^m\in S^+_X$ satisfying $\|\hat{y}_n^m\pm  y_{n+1}^m\|\le1+\varepsilon_n^m$. This concludes the construction of the sequence $(y_n^m)$.
    
    Since $B_{X^{**}}$ is weak$^*$ compact, then there exists a subnet $\{y_\alpha^m:\alpha\in\mathscr A_m\}$ of $(y_n^m)$ which is weak$^*$ convergent to some point $y_m^{**}\in X^{**}$. Define $x_{\alpha,\beta}^m:=y^m_\alpha-y_\beta^m$, where $\alpha,\beta\in\mathscr A_m$ and notice that
    \begin{align*}
    1 & =\|y_0^m\|\le\left\|\sum_{i=0}^{n+1}y_i^m\right\|=\left\|\left\|\sum_{i=0}^ny_i^m\right\|\cdot \hat{y}_n^m+y_{n+1}^m\right\| \le(1+\varepsilon_n^m)\left\|\sum_{i=0}^ny_i^m\right\|\le\ldots\\
    & \le\prod_{i=1}^n(1+\varepsilon_i^m)\le1+m^{-1}.
    \end{align*}
    holds for all $n\in\mathbb N$. Moreover,
    \begin{equation*}
        \|y\pm x_{\alpha,\beta}^m\|=\|y\pm(y_{n_\alpha}^m-y_{n_\beta}^m)\|\le\left\|\sum_{i=0}^{\max\{n_\alpha,n_\beta\}}y_i^m\right\|\le1+m^{-1}
    \end{equation*}
    and
    \begin{equation*}
        \|x_{\alpha,\beta}^m\|=\|y_{n_\alpha}^m-y_{n_\beta}^m\|\le\left\|\sum_{i=0}^{\max\{n_\alpha,n_\beta\}}y_i^m\right\|\le1+m^{-1}
    \end{equation*}
    hold for all $\alpha,\beta\in\mathscr A_m$ satisfying $n_\alpha\not=n_\beta$. Moreover, it is clear that $\|y^m_\alpha\pm y^m_\beta\|\le1+m^{-1}$ holds too, therefore
    \begin{equation*}
        2=\|2y_\alpha^m\|\le\|y_\alpha^m+y_\beta^m\|+\|y_\alpha^m-y_\beta^m\|\le\|y_\alpha^m+y_\beta^m\|+1+m^{-1}.
    \end{equation*}
    From this, and an analogous argument, we conclude that $\|y^m_\alpha\pm y^m_\beta\|\ge1-m^{-1}$. Similarly one can prove that $\|y\pm x^m_{\alpha,\beta}\|\ge1-m^{-1}$ too. To sum up, so far we proved that
    \begin{equation}\label{Equation 1}
        (\forall\alpha,\beta\in\mathscr A: n_\alpha\not=n_\beta)\hspace{0.1cm}|\|y\pm x^m_{\alpha,\beta}\|-1|\le m^{-1}
    \end{equation}
    and
    \begin{equation}\label{Equation 2}
        (\forall\alpha,\beta\in\mathscr A: n_\alpha\not=n_\beta)\hspace{0.1cm}|\|x_{\alpha,\beta}^m\|-1|\le m^{-1}.
    \end{equation}
    
    Now we proceed to construct the desired net. For this purpose, let us call $\mathscr U$ the neighborhood filter around $0$ in $X$ endowed with the weak topology and let us choose, for every $(m,U)\in\mathbb N\times\mathscr U$, some $x_{m,U}\in U$, where $x_{m,U}=x^m_{\alpha,\beta}$ for some $\alpha,\beta\in\mathscr A$ satisfying $n_\alpha\not=n_\beta$. Notice that such $\alpha$ and $\beta$ can be found because, for all $m\in\mathbb N$, $x^m_{\alpha,\beta}=y^m_\alpha-y^m_\beta$, where the net $(y^m_\alpha)$ is weak$^*$ convergent to $y^{**}_m$, and the map $\alpha\rightarrow n_\alpha$ is both cofinal and increasing. Now notice that $\mathbb N\times\mathscr U$ is a directed set if endowed with the natural preorder
    \begin{equation*}
        (m,U)\le(m',U'):\iff m\le m'\text{ and }U'\subset U.
    \end{equation*}
    
    It is easy to verify that the net $\{x_{m,U}:(m,U)\in\mathbb N\times\mathscr U\}$ is weakly null and, thanks to (\ref{Equation 1}) and (\ref{Equation 2}), satisfies
    \begin{equation*}
        \lim\|x_{m,U}\|=1\text{ and }\lim\|y\pm x_{m,U}\|=1.
    \end{equation*}
\end{proof}

Notice that the assumption of $X$ being $\positiveLASQ$ and the lattice structure over $X$ are necessary only during the construction of the sequence $(y_n^m)$. So the following question arises naturally.

\begin{question}\label{Question: Main result without lattice}
    Let $X$ be a Banach space. If is $\LASQ$, then does it imply that, for every $y\in S_X$, there exists a net $(x_\alpha)$ in $X$ satisfying the conditions of Theorem~\ref{Theorem: Main result}?
\end{question}

The following criterion is obtained through a few modifications in the proof of Theorem~\ref{Theorem: Main result}.

\begin{theorem}\label{Theorem: LASQ implies WASQ}
    Let $X$ be a Banach lattice such that $X^*$ is separable. If $X$ is $\positiveLASQ$, then $X$ is $\WASQ$.
\end{theorem}
\begin{proof}
    Since $X^*$ is separable, then $B_{X^{**}}$, with respect to the weak$^*$ topology, is metrizable. Hence, in the proof of Theorem~\ref{Theorem: Main result}, we can replace the subnet $(y_\alpha^m)$ with a subsequence $(y_{n_k}^m)$ and simply define $x_{k}^m:=y_{n_k}^m-y_{n_{k-1}}^m$ which is clearly weakly null. Let now $(x^*_n)$ be a dense sequence in $X^*$. For every $m\in\mathbb N$ there exists $n_m\in\mathbb N$ satisfying $|x^*_i(x^m_n)|\le m^{-1}$ for all $1\le i\le m$ and $n\ge n_m$. Then the sequence $(x_{n_m}^m)$ is weakly null and, thanks to (\ref{Equation 1}) and (\ref{Equation 2}), satisfies
    \begin{equation*}
        \lim\|x^m_{n_m}\|=1\text{ and }\lim\|y\pm x^m_{n_m}\|=1.
    \end{equation*}
\end{proof}

\section{The diameter two property}\label{Section: The diameter two property}

It is known that if a Banach space is $\LASQ$, then it has the $\LDtwoP$, and, if a Banach space is $\WASQ$, then it has the $\DtwoP$. Thanks to Theorem~\ref{Theorem: Main result} we can extend these results.

\begin{theorem}\label{Theorem: LASQ implies D2P}
    Let $X$ be a Banach lattice. If $X$ is $\positiveLASQ$, then $X$ has the $\DtwoP$.
\end{theorem}
\begin{proof}
    Let $U$ be a non-empty relatively weakly open set in $B_X$ and $\varepsilon>0$. Notice that
    \begin{equation*}
        U_\delta:=\bigcap_{i=1}^n\{x\in B_X:x^*_i(x)\ge1-\delta\}\subset U
    \end{equation*}
    for some $0<\delta\le\varepsilon$ and $x^*_1,\ldots,x^*_n\in S_{X^*}$. Since $X$ is $\LASQ$, then it is infinite dimensional, hence we can find $y\in U_{\delta/2}\cap S_X$ and construct the net $(x_\alpha)$ as in Theorem~\ref{Theorem: Main result} with respect to $y$. It is now clear that we can find $\alpha$ satisfying $|x^*_i(x_\alpha)|\le\delta/2$ for all $1\le i\le n$ and
    \begin{equation*}
        1-\delta/2\le\|x_\alpha\|,\|y\pm x_\alpha\|\le1+\delta/2.
    \end{equation*}
    Therefore $\|y+x_\alpha-(y-x_\alpha)\|=2\|x_\alpha\|\ge2-\delta\ge2-\varepsilon$ and $x^*_i(y\pm x_\alpha)\ge1-\delta/2-\delta/2\ge1-\varepsilon$ holds for all $1\le i\le n$, hence, up to a perturbation argument, $y\pm x_\alpha\in U$.
\end{proof}

Notice that the converse of Theorem~\ref{Theorem: LASQ implies D2P} doesn't hold. In fact $L_1[0,1]$, as already noted, is not $\positiveLASQ$, but it is $\WASQ$ and therefore has the $\DtwoP$.

In the rest of the section, thanks to Theorem~\ref{Theorem: LASQ implies D2P}, we will extend some results already existing in the literature which provide new examples of spaces enjoying the $\DtwoP$.

\begin{corollary}\label{Corollary: Köthe-Bochner spaces}
    Let $S$ be a complete, $\sigma$-finite measure space, $E$ a Köthe function space over $S$ and X a Banach lattice. If $X$ is $\positiveLASQ$, then the Köthe-Bochner space $E(X)$ is $\positiveLASQ$ and therefore has the $\DtwoP$.
\end{corollary}
\begin{proof}
    Notice that $E(X)$ is a Banach lattice in a natural way, so the claim is meaningful.

    The proof is just an adaptation of \cite[Theorem~3.1]{H2021-preprint}. So in the following we will point out only the small changes needed to make the proof work.
    \begin{enumerate}
        \item Find the sequence $(z_n)$ in $S^+_X$.
        \item Consider the set $F'(s):=F(s)\cap S^+_Z$ instead of $F(s)$.
        \item Define $g(s,z)$ for $S\in S'$ and $z\in S^+_Z$.
    \end{enumerate}
    With these choices the selection $\tilde f$ is hence defined from $S'$ to $S_Z^+$ and therefore $h\in S^+_{E(X)}$.
\end{proof}

The next result concerns ultrapowers. Let us recall that the ultrapower of a Banach lattice is still a Banach lattice if we take as positive cone the elements that admit representatives such that each coordinate is positive \cite[Page~55]{JL2001}. Keeping this in mind, an obvious adaptation of \cite[Proposition~4.2]{H2018} proves the following.

\begin{corollary}\label{Corollary: ultrapowers}
    Let $X$ be a Banach lattice and $\mathscr F$ a non-principal ultrafilter in $\mathbb N$. If $X$ is $\positiveLASQ$, then the ultrapower $X^\mathscr F$ is $\positiveLASQ$ and therefore has the $\DtwoP$.
\end{corollary}

Before stating the last result, let us introduce some notation. Given $y\in S_X$, define $n(X,u)$ as the largest non-negative real number $r$ satisfying
\begin{equation*}
    rx\le\sup\{|x^*(x)|:x^*\in S_X\text{ and }x^*(y)=1\}.
\end{equation*}

\begin{corollary}\label{Corollary: projective tensor product}
    Let $X$ be a Banach lattice and $Y$ an AL-space such that there is $y^*\in S_{Y^*}$ satisfying $n(Y^*,y^*)=1$. If $X$ is $\positiveASQ$, i.e., for every $y_1,\ldots,y_n\in S_X$ and $\varepsilon>0$, there exists $x\in S^+_X$ satisfying $\|y_i\pm x\|\le1+\varepsilon$ for $1\le i\le n$, then $X\hat{\otimes}_\pi Y$ is $\positiveLASQ$ and therefore has the $\DtwoP$.
\end{corollary}
\begin{proof}
    Since $Y$ is an AL-space, then $X\hat{\otimes}_\pi Y$ is a Banach lattice \cite[Theorem~2B]{F1974}.

    The rest of the proof is an adaptation of \cite[Proposition~2.11]{LLR2017}. So let us point out the main differences.
    \begin{enumerate}
        \item Thanks to Remark~\ref{Remark: equivalent definition LASQ+}, we only need to consider the case $u\in S^+_{X\hat{\otimes}_\pi Y}$.
        \item Since $\conv(S^+_X\otimes S^+_Y)$ is dense in $S^+_{X\hat{\otimes}_\pi Y}$ \cite[1E~(ii)]{F1974}, we can find $v=\sum_{i=1}^n\lambda_ix_i\otimes y_i$ such that $\|u-v\|\le\varepsilon$, $x_i\in S^+_X$ and $y_i\in S^+_Y$ for $1\le i\le n$.
        \item As $X$ is $\positiveASQ$, we can find $x\in S^+_X$ such that $\|x_i\pm x\|\le1+\varepsilon$ holds for $1\le i\le n$.
        \item With these choices, $z:=x\otimes\sum_{i=1}^n\lambda_iy_i\in S^+_X$ thanks to \cite[1E~(ii)]{F1974} and $\|u\pm z\|\le1+2\varepsilon$.
    \end{enumerate}
\end{proof}
%==========================================================

%==========================================================
\end{document}